\documentclass[english]{amsart}

\usepackage{amsmath,amssymb,enumerate}
\usepackage{amscd,amsxtra,calc}

\usepackage[T1]{fontenc}
\usepackage[all]{xy}

\usepackage{babel}
\usepackage{amstext}
\usepackage{amsmath}
\usepackage{amsfonts}
\usepackage{latexsym}
\usepackage{ifthen}
\usepackage{verbatim}
\usepackage{enumerate}

\usepackage{xypic}
\xyoption{all}
\newcommand{\id}{{\rm id}}

\newtheorem{lemma1}{}[section]

\newenvironment{lemma}{\begin{lemma1}{\bf Lemma.}}{\end{lemma1}}
\newenvironment{example}{\begin{lemma1}{\bf Example.}\rm}{\end{lemma1}}

\newenvironment{theorem}{\begin{lemma1}{\bf Theorem.}}{\end{lemma1}}

\newenvironment{remark}{\begin{lemma1}{\bf Remark.}\rm}{\end{lemma1}}
\newenvironment{definition}{\begin{lemma1}{\bf Definition.}}{\end{lemma1}}

\newenvironment{problem}{\begin{lemma1}{\bf Problem.}}{\end{lemma1}}

\newenvironment{remark*}{{\bf Remark.}}{}
\newenvironment{example*}{{\bf Example.}}{}

\newcommand{\C}{\ensuremath{\mathbb{C}}}

\newcommand{\PP}{\ensuremath{\mathbb{P}}}

\newcommand{\holom}[3]{\ensuremath{#1:#2  \rightarrow #3}}
\newcommand{\fibre}[2]{\ensuremath{#1^{-1} (#2)}}

\makeatletter
\ifnum\@ptsize=0  \fi \ifnum\@ptsize=2  \fi 

\newcommand\sF{{\mathcal F}}

\newcommand\sK{{\mathcal K}}

\newcommand\sO{{\mathcal O}}

\newcommand{\upA}{\ensuremath{\tilde{A}}}

\usepackage{xcolor}

\newcommand{\PX}{\PP(T_X)}

\title{Normal split submanifolds of rational homogeneous spaces}
\date{October 24, 2022}

\author {Enrica Floris}
\author {Andreas H\"oring}

\address{Enrica Floris, Universit\'e de Poitiers, Laboratoire de Math\'ematiques et Applications,\linebreak UMR~CNRS 7348, T\'el\'eport 2, Boulevard Marie et Pierre Curie, BP 30179, 86962 Futuroscope Chasseneuil Cedex, France}
\email{enrica.floris@univ-poitiers.fr}

\address{Andreas H\"oring, Universit\'e C\^ote d'Azur, CNRS, LJAD, France, Institut universitaire de France}
\email{Andreas.Hoering@univ-cotedazur.fr}

\begin{document}

\begin{abstract}
Let $M \subset X$ be a submanifold of a rational homogeneous space $X$ such that
the normal sequence splits. We prove that $M$ is also rational homogeneous.
\end{abstract}

\maketitle


\section{Introduction}

Let $M \subset \PP^n$ be a projective manifold such that the normal sequence
$$
0 \rightarrow T_M \rightarrow T_{\PP^n} \otimes \sO_M \rightarrow N_{M/\PP^n} \rightarrow 0
$$
splits. Since the tangent bundle of the projective space is ample, the existence of a splitting map $T_{\PP^n} \otimes \sO_M \twoheadrightarrow T_M$ yields that $T_M$ is also ample, hence $M \simeq \PP^{\dim M}$ by Mori's theorem \cite{Mor79}.
In fact a theorem of van de Ven \cite{vdv59} states that $M \subset \PP^n$ is a linear subspace. It is natural to expect that similar statements exist for submanifolds of arbitrary homogeneous spaces. The case of tori has been solved by Jahnke \cite[Thm.3.2]{Jah05}, so we consider the following

\begin{problem}
Let $X \simeq G/P$ be a rational homogeneous space. Let $i: M \hookrightarrow X$ be a submanifold such that the normal sequence
\begin{equation}
\label{normal}
0 \rightarrow T_M \stackrel{\tau_i}{\rightarrow} T_X \otimes \sO_M \rightarrow N_{M/X} \rightarrow 0
\end{equation}
splits, i.e. there exists a morphism $\holom{s_M}{T_X \otimes \sO_M}{T_M}$ such that $s_M \circ \tau_i = \id_{T_M}$.
\begin{itemize}
\item Describe $M$ up to isomorphism.
\item Describe the embedding $i: M \hookrightarrow X$.
\end{itemize}
\end{problem}

Both problems have been solved for special classes of rational homogeneous spaces like hyperquadrics, Grassmannians \cite[Thm.4.7, Prop.5.2]{Jah05}, irreducible compact Hermitian symmetric spaces \cite[Thm.2.3]{Din22} or certain blow-ups of the projective space \cite{Jah05, Li22}.
In this paper we focus on the first problem and prove a structure result that holds without further assumptions on $X$,
thereby improving \cite[Prop.2.2]{Jah05} and \cite[Thm.2.1]{Din22}:

\begin{theorem} \label{theorem-split}
Let $X \simeq G/P$ be a rational homogeneous space, and let $M \subset X$ be a  submanifold that is normal split (cf. Definition \ref{definition-normal-split}). Then $M$ is rational homogeneous.
\end{theorem}

Since the vector bundle $T_X \otimes \sO_M$ is globally generated and $M \subset X$ is normal split, the tangent bundle of $T_M$ is globally generated. Thus $M$ is homogeneous and by \cite[Satz 1]{BR62} we have
$$
M \simeq A \times Y
$$
where $A$ is an abelian variety and $Y$ is rational homogeneous.
Thus in order to prove Theorem \ref{theorem-split} we can assume without loss of generality that $M \simeq A$ is abelian (cf. Lemma \ref{lemma-reductions}). Since the tangent bundle of the rational homogeneous space $T_X$ is big \cite{Ric74} and the tangent bundle of an abelian variety is not, one is tempted to argue as in the case of the projective space. Note however that the quotient of a big vector bundle is in general not big, so a priori neither $T_X \otimes \sO_A$ nor its quotient $T_A$ are big. 
Therefore our argument proceeds in a different fashion:
denote by 
$$
\holom{\pi}{\PP(T_X)}{X}
$$ 
the Grothendieck projectivisation of the tangent bundle. The global sections of $T_X$ define a morphism
\begin{equation}
\label{definition-varphi}
\holom{\varphi_X}{\PP(T_X)}{\PP(H^0(X, T_X))}
\end{equation}
that is generically finite onto its image. Since $A$ is abelian, the splitting map allows to define
a lifting $A \rightarrow \PX$ such that the image $\upA$ is contracted by $\varphi_X$.
If the fibres of $\varphi_X$ are smooth and rationally connected the existence of $\upA$ is often enough to obtain a contradiction, cp. \cite{Jah05}.
For an arbitrary rational homogeneous space $X$ the fibres of $\varphi_X$ are not necessarily smooth
(cf. also Remark \ref{remark-fibres-varphi}), but we show in Lemma \ref{lemma-key}
that manifolds contracted by $\varphi_X$ are always integral submanifolds with respect to the contact structure on $\PX$. Theorem \ref{theorem-split} then follows without too much difficulty.

Our result can be extended to normal split submanifolds of projective manifolds such that the tangent bundle $T_X$ is nef and big, cf. Remark \ref{remark-nef-and-big}. By a well-known conjecture of Campana and Peternell \cite{CP91}, these are exactly the rational homogeneous spaces.

{\bf Acknowledgements.} The authors thank B. Fu, J. Liu, L. Manivel, B. Pasquier and R. \'Smiech for answering their ignorant questions about homogeneous spaces and contact manifolds.
The second-named author thanks the Institut Universitaire de France for providing excellent working conditions.

\section{Notation and basic facts}

We work over the complex numbers, for general definitions we refer to \cite{Har77}. 
Varieties will always be supposed to be irreducible and reduced. 
We use the terminology of \cite{Deb01, KM98}  for birational geometry and notions from the minimal model program. We follow \cite{Laz04a} for algebraic notions of positivity.

Given a morphism $\holom{\varphi}{A}{B}$ between complex manifolds we denote by
$$
\holom{\tau_\varphi}{T_A}{\varphi^* T_B}
$$
the tangent map. For a submanifold $M \hookrightarrow X$ of  a complex manifold $X$, we will simply denote the tangent map by
$$
\holom{\tau_M}{T_M}{T_X \otimes \sO_M}.
$$

\begin{definition} \label{definition-normal-split}
Let $X$ be a complex manifold, and let $\holom{i}{M}{X}$ be a submanifold. We say that
the $M$ is normal split in $X$ if the inclusion 
$$
0 \rightarrow T_M \stackrel{\tau_i}{\rightarrow} T_X \otimes \sO_M
$$ 
admits a splitting morphism $s_M: T_X \otimes \sO_M \rightarrow T_M$ such that $s_M \circ \tau_i = \id_{T_M}$.
\end{definition}

\begin{remark*} While $\tau_i$ is determined by the embedding, 
the splitting morphism $s_M$ is not unique. Our statements will not depend on the choice of $s_M$.
\end{remark*}

\begin{lemma} \label{lemma-reductions}
Let $X \simeq G/P$ be a rational homogeneous space, and let $M \subset X$ be a submanifold that is normal split.
If $M$ is not rational homogeneous, there exists an abelian
variety  of positive dimension  $A \subset X$ that is normal split.
\end{lemma}

\begin{proof}
Since $M$ is homogeneous, we have by \cite[Satz 1]{BR62} that
$M \simeq A \times Y$ with $A$ an abelian variety and $Y$ rational homogeneous. Since $M$ is not rational homogeneous, we have $\dim A>0$. Yet $M$ is a product, so
the abelian variety $A$ is normal split in $M$. Thus by \cite[1.2.2]{Jah05} the abelian variety $A$ is normal split in $X$. 
\end{proof}

\section{Geometry of  the projectivised tangent bundle}

Let $X$ be a complex manifold, denote by $\holom{\pi}{\PP(T_X)}{X}$ the Grothendieck projectivisation of its tangent bundle.
Let 
$$
0 \rightarrow \sO_{\PP(T_X)} \rightarrow \pi^* \Omega_X \otimes \sO_{\PP(T_X)}(1) \rightarrow 
T_{\PP(T_X)/X} \rightarrow 0
$$
be the relative Euler sequence. Dualising and tensoring by $\sO_{\PP(T_X)}(1)$ we obtain the canonical quotient map
\begin{equation} \label{definition-q}
\holom{q}{\pi^* T_X}{\sO_{\PP(T_X)}(1)}.
\end{equation}
The composition of the tangent map
$\holom{\tau_\pi}{T_{\PP(T_X)}}{\pi^* T_X}$
with the canonical quotient map $q$ defines an exact sequence
\begin{equation} \label{define-contact}
0 \rightarrow \sF \rightarrow T_{\PP(T_X)} \stackrel{q \circ \tau_\pi}{\rightarrow} \sO_{\PP(T_X)}(1) \rightarrow 0.
\end{equation}
The corank one distribution $\sF \subset T_{\PP(T_X)}$ is a contact distribution \cite[Sect.13.2]{Bla10}, i.e. the map 
$$
\bigwedge^2 \sF \rightarrow T_{\PP(T_X)}/\sF \simeq \sO_{\PP(T_X)}(1)
$$
induced by the Lie bracket on $\PX$ is surjective.

Assume now that $X \simeq G/P$ is rational homogeneous,
so the tautological bundle $\sO_{\PP(T_X)}(1)$ is nef and big \cite{Ric74}. 

Since $T_X$ is globally generated, the tautological line bundle
$\sO_{\PP(T_X)}(1)$ is globally generated and defines a morphism
$$
\holom{\varphi_{|\sO_{\PP(T_X)}(1)|}}{\PP(T_X)}{\PP(H^0(X, T_X))}
$$
such that $\sO_{\PP(T_X)}(1) \simeq \varphi_{|\sO_{\PP(T_X)}(1)|}^* \sO_{\PP(H^0(X, T_X))}(1)$.
We denote by
\begin{equation} \label{define-varphiX}
\holom{\varphi_X}{\PX}{Y}
\end{equation}
the Stein factorisation of $\varphi_{|\sO_{\PP(T_X)}(1)|}$
and by $L$ the pull-back of $\sO_{\PP(H^0(X, T_X))}(1)$
to $Y$. By construction $L$ is ample and $\sO_{\PP(T_X)}(1)  \simeq \varphi_X^* L$.

\begin{remark} \label{remark-fibres-varphi}
Since $\sO_{\PP(T_X)}(1)$ is nef and big, the morphism 
$\varphi_X$ is birational.
The canonical bundle of $\PX$ is isomorphic to $\sO_{\PP(T_X)}(-n)$,
so it is trivial on the fibres of $\varphi_X$.
Thus we have 
$$
K^*_Y  \simeq L^{\otimes n}
$$
and $Y$ is a Fano variety with canonical Gorenstein
singularities. In fact it is not difficult to see that $Y$ has an induced singular contact structure \cite{CF02}
given by a global section of $H^0(Y, \Omega_Y^{[1]} \otimes L)$.

We can apply
\cite[Thm.2]{Kaw91} to see that the irreducible components of $\varphi_X$-fibres are uniruled. We also know by \cite[Cor.1.5]{HM07} that the fibres are rationally chain-connected. Note however that this does not imply that the irreducible components are rationally chain-connected \cite[p.119]{HM07}
\end{remark}

The following statement, inspired by \cite[Prop.5.9]{MOSWW} is certainly well-known to experts, we give a detailed proof for the convenience of the reader:

\begin{lemma} \label{lemma-key}
Let $X \simeq G/P$ be a rational homogeneous space, and let 
$\holom{\pi}{\PP(T_X)}{X}$ be its projectivised cotangent bundle.
Let $F \subset \PP(T_X)$
be a smooth quasi-projective subvariety that is contracted by the birational map
\eqref{define-varphiX} onto a point. Then $F$ is an integral variety with respect to the contact distribution $\sF \subset T_{\PP(T_X)}$, i.e. one has
$$
T_{F} \subset \sF \otimes \sO_{F}.
$$
\end{lemma}

We will see that this lemma is a translation of the fact that fibres of a symplectic resolution
are isotropic with respect to the symplectic form \cite[Thm.1.2]{Wie03}, \cite{Nam01}, a strategy that appears in the literature at several places, e.g. \cite{Bea00, SW04}. 

Let us recall the relation
between the two setups:
let $\mathcal Y \rightarrow Y$ be the total space of $L^*$ with the zero section removed, and set $\mathcal X := \PP(T_X) \times_Y \mathcal Y$. Denote the natural maps by
$$
\holom{\tilde \varphi_X}{\mathcal X}{\mathcal Y}, \qquad \holom{\tau}{\mathcal X}{\PP(T_X)}.
$$
Then $\mathcal X \rightarrow \PP(T_X)$ is a $\C^*$-bundle and in fact the total space of $\sO_{\PP(T_X)}(-1)$ with the zero section removed \cite[Lemma 4.1]{Fu}. The spaces $\mathcal X$ and $\mathcal Y$
are symplectic and $\tilde \varphi_X$ is a symplectic resolution \cite[Lemma 4.2]{Fu}. More precisely the contact form on $\PP(T_X)$ is obtained from the symplectic form $\tilde \omega$ 
on $\mathcal X$ by contracting with a vector field generated by the $\C^*$-action (cf. proof of \cite[Lemma 4.2]{Fu}).

\begin{proof}[Proof of Lemma \ref{lemma-key}]
By \eqref{define-contact}
the contact distribution $\sF$ is the kernel of the contact map 
$$
T_{\PP(T_X)} \rightarrow \sO_{\PP(T_X)}(1) \rightarrow 0.
$$
We identify the contact map with a section  
$$
\theta \in H^0(\PP(T_X), \Omega_{\PP(T_X)} \otimes \sO_{\PP(T_X)}(1)).
$$
Since $\sO_{\PP(T_X)}(1) \simeq \varphi_X^* L$ we can find an analytic neighbourhood 
$U$ of $\fibre{\varphi_X}{\varphi_X(F)}$ such 
that $\sO_{\PP(T_X)}(1)$ is trivial on $U$.
Hence
$$
\theta|_U \in H^0(U, \Omega_U)
$$
and we are done if we show that
$$
\theta|_F \in H^0(F, N_{F/\PP(T_X)}^*),
$$
where $N_{F/\PP(T_X)}^* \subset \Omega_U \otimes \sO_F$ is the conormal bundle of $F$.

{\em Proof of the claim.} We will reduce the claim to the corresponding statement for the symplectic forms: 
let $T \simeq \C^*$ be the fibre of $\mathcal Y \rightarrow Y$ over $\varphi_X(F)$. Then $F \times_Y T \simeq F \times T \subset \mathcal X$ and we set 
$$
\holom{\sigma}{F \times T}{T},
$$
where $\sigma$ is the restriction of $\tilde \varphi_X$ to $F \times T$.
Then by \cite[Lemma 2.9]{Kal} there exists a dense open 
subset $T_0 \subset T$ and a holomorphic two-form $\omega_T$ on $T_0$ such that
$$
\sigma^* \omega_T = \eta|_{F \times T_0}
$$
where $\eta$ is the image of $\tilde \omega$ under the restriction map
$$
H^0(\mathcal X, \Omega^2_{\mathcal X}) \rightarrow
H^0(F \times T, \Omega^2_{\mathcal X} \otimes \sO_{F \times T}) \rightarrow
H^0(F \times T, \Omega^2_{F \times T}).
$$
Since $T_0$ is a curve, the holomorphic two-form $\omega_T$ is zero, hence $\eta=0$. By \cite[II, Ex.5.16]{Har77} the conormal sequence
$$
0 \rightarrow N_{F \times T/\mathcal X}^* \rightarrow \Omega_{\mathcal X}  \otimes \sO_{F \times T} \rightarrow
\Omega_{F \times T} \rightarrow 0
$$
induces a filtration of the kernel $\sK$ of the surjection $\Omega^2_{\mathcal X} \otimes \sO_{F \times T}
\rightarrow \Omega^2_{F \times T}$:  
\begin{equation} \label{filtrate}
0 \rightarrow \bigwedge^2 N_{F \times T/\mathcal X}^*
\rightarrow \sK \rightarrow
N_{F \times T/\mathcal X}^* \otimes \Omega_{F \times T} 
\rightarrow 0.
\end{equation}
By what precedes we know that 
$$
\tilde \omega|_{F \times T} \in H^0(F \times T, \sK),
$$
we will now deduce $\theta|_F \in H^0(F, N_{F/\PP(T_X)}^*)$: by the discussion before the proof $\theta|_F$
is obtained from $\tilde \omega|_{F \times T}$ by contracting with a vector field $v$ generated by the $\C^*$-action. 
Since this vector field is mapped onto zero by $\tau$, the contraction with a $2$-form that is a pull-back from $F$ is equal to zero. Since 
$$
N_{F \times T/\mathcal X}^* = \tau^* N_{F/\PP(T_X)}^*
$$
we obtain from \eqref{filtrate} that the contraction map 
$\sK \stackrel{\lrcorner v}{\longrightarrow} \Omega_{\PP(T_X)}|_F$
factors through a morphism
$$
N_{F \times T/\mathcal X}^* \otimes \Omega_{F \times T}  \stackrel{\lrcorner v}{\longrightarrow} 
 \Omega_{\PP(T_X)}|_F.
$$
Yet
$$
\Omega_{F \times T} \simeq \tau^* \Omega_F \oplus \varphi_X^* \Omega_T,
$$
so if we decompose $\tilde \omega|_{F \times T}=\tilde \omega_1 + \tilde \omega_2$
according to the direct sum
$$
N_{F \times T/\mathcal X}^* \otimes \Omega_{F \times T}  
\simeq
\left( \tau^* N_{F/\PP(T_X)}^* \otimes \tau^* \Omega_F \right)
\oplus
\left( \tau^* N_{F/\PP(T_X)}^* \otimes \varphi_X^* \Omega_T \right)
$$
we see that $\tilde \omega_1 \lrcorner v=0$ while 
$$
 \tilde \omega_2 \lrcorner v \in H^0(F, N_{F/\PP(T_X)}^*).
$$
Since $\theta|_F =(\tilde \omega|_{F \times T})  \lrcorner v =  \tilde \omega_2 \lrcorner v$ this shows the claim.
\end{proof}

The following example shows that the crucial point in Lemma \ref{lemma-key} is that the contact form on $\PX$ is a reflexive pull-back from the singular space $Y$, i.e. we use
that $T_X$ is nef {\em and big}.

\begin{example}
Let $X=\C^2/\Lambda$ be an abelian surface, so $X$ is homogeneous and the natural
map \eqref{define-varphiX} is given by the projection
$$
\holom{\varphi_X}{\PX \simeq X \times \PP^1}{\PP^1}.
$$
We will now follow the notation of \cite[Sect.13.2]{Bla10} for the local computation of the contact form $\theta$:
for linear coordinates $z_1, z_2$ on $\C^2$ the  contact form on $\PX$ is
$$
\theta = \sum_{i=1}^2 d z_i \otimes \zeta_i
$$
where $\sum_{i=1}^2 \zeta_i dz_i$ are fibrewise coordinates on $\PP(T_{X,x}) \simeq {\bf P}(\Omega_{X,x})$ (where ${\bf P}(\Omega_{X,x})$ is the space of lines in $\Omega_{X,x}$).

Assume now that $A \subset X$ is an elliptic curve corresponding to the linear subspace $\C z_1 \subset \C^2$,
then 
$$
\frac{\partial}{\partial z_1} \mapsto \frac{\partial}{\partial z_1},
\qquad
\frac{\partial}{\partial z_2} \mapsto 0
$$
defines a splitting $T_X \otimes \sO_A \rightarrow T_A$ of the tangent map $\tau_A$.
The curve $\upA \subset \PX$ is contained in
$X \times {\bf P}(dz_1) \subset X \times {\bf P}(\Omega_{X})$, so
it is contracted by $\varphi_X$.
The restriction of $\theta$ to $X \times {\bf P}(dz_1)$ is simply
the form $dz_1$, in particular the composition
$$
T_{\upA} \hookrightarrow T_{\PX} \otimes \sO_{\upA}
\stackrel{v \mapsto dz_1(v)}{\longrightarrow} \sO_{\upA}
$$
is surjective.
\end{example}

We make a basic observation:

\begin{lemma} \label{lemma-observation}
Let $X$ be a projective manifold, and let 
$\holom{\pi}{\PP(T_X)}{X}$ be its projectivised cotangent bundle.
Let $A \subset X$ be an abelian variety that
is normal split with splitting map $\holom{s_A}{T_X \otimes \sO_A}{T_A}$ and fix a quotient  $\holom{q_A}{T_A}{\sO_A}$.
Let $\holom{\sigma_A}{A}{\PP(T_X)}$ be the lifting determined by the 
quotient line bundle
$$
\holom{q_A \circ s_A}{T_X \otimes \sO_A}{\sO_A},
$$
and denote by $\upA \subset \PX$ its image. Since $\upA$ maps isomorphically onto its image in $X$, 
we can consider the tangent morphism
$$
\tau_A: T_{\tilde A} \simeq T_A \rightarrow T_X \otimes \sO_A \simeq \pi^* (T_X) \otimes \sO_{\upA}.
$$ 
Then the composition with the canonical quotient map \eqref{definition-q} gives a surjective map
$$
q \circ \tau_A: T_{\tilde A}  \simeq T_A \rightarrow \pi^* (T_X) \otimes \sO_{\upA} \rightarrow 
\sO_{\PP(T_X)}(1) \otimes \sO_{\upA}.
$$
\end{lemma}

\begin{proof}[Proof of Lemma \ref{lemma-observation}]
The lifting $\sigma_A$ is determined by the quotient line bundle
$\holom{q_A \circ s_A}{T_X \otimes \sO_A}{\sO_A}$,
so by the universal property of the projectivisation \cite[App.A]{Laz04a} 
the pull-back of the canonical quotient  map $\holom{q}{\pi^* T_X}{\sO_{\PP(T_X)}(1)}$ via $\sigma_A$ identifies to $q_A \circ s_A$.
Since $s_A$ is a splitting map for $\holom{\tau_A}{T_A}{T_X \otimes \sO_A}$,
the composition $q_A \circ s_A \circ \tau_A = q_A$ is surjective.
\end{proof}

\begin{remark} \label{remark-lift-rational}
It is instructive to compare the situation with liftings of rational curves:
let $X$ be a smooth quadric surface, and let $l \subset X$ be a line of a ruling.
Then $l \subset X$ is normal split: we have 
$$
T_X \otimes \sO_l \simeq T_l \oplus \sO_l,
$$
so the trivial quotient $T_X \otimes \sO_l \rightarrow \sO_l$ determines a lifting
of $l$ to $\PP(T_X)$ such that the image $\tilde l$ is contained in a $\varphi_X$-fibre. 
However the morphism
$$
T_{\tilde l} \rightarrow \pi^* (T_X) \otimes \sO_{\tilde l} \rightarrow \sO_{\tilde l}
$$
is not surjective: we have $T_{\tilde l} \simeq \sO_{\PP^1}(2)$, so any morphism to a trivial bundle must vanish. The difference to Lemma \ref{lemma-observation}
is that the trivial quotient $T_X \otimes \sO_l \rightarrow \sO_l$ does not factor through 
a morphism $T_X \otimes \sO_l \rightarrow T_l$.
\end{remark}

\begin{proof}[Proof of Theorem \ref{theorem-split}]
We argue by contradiction and assume that the submanifold $M \subset X$ is not rational homogeneous. By Lemma \ref{lemma-reductions} we can assume without loss of generality
that $M \simeq A$ is an abelian variety. 
We fix a splitting map $\holom{s_A}{T_X \otimes \sO_A}{T_A}$ and a trivial quotient $\holom{q_A}{T_A}{\sO_A}$.
Denote by $\upA \subset \PX$ the lifting of $A$ to $\PX$ determined by the quotient
$q_A \circ s_A$. By Lemma \ref{lemma-observation}
the map
$$
q \circ \tau_A: T_{\tilde A} \rightarrow
\sO_{\PP(T_X)}(1) \otimes \sO_{\upA}
$$
is surjective. Since $\upA \subset \PX$ the tangent map $\tau_A$ factors through
through the tangent map
$$
\tau_{\upA} : T_{\upA} \rightarrow T_{\PX} \otimes \sO_{\upA},
$$
and we have a commutative diagram
$$
\xymatrix{
T_{\upA} \ar[rr]^{\tau_{\upA}} \ar[d]^{\simeq} & & T_{\PX} \otimes \sO_{\upA} \ar[d]^{\tau_\pi} &
\\
T_A  \ar[rr]^{\tau_{A}} & &  \pi^* T_X \otimes \sO_A  \ar[r]^q &  \sO_{\PP(T_X)}(1) \otimes \sO_{\upA}
}
$$
By construction the contact map \eqref{define-contact}
is the composition
$q \circ \tau_\pi$. Since $q \circ \tau_A$ is surjective, we obtain that $q \circ \tau_\pi \circ \tau_{\upA}$
is surjective, in particular $\upA \subset \PX$ is not integral with respect to the contact structure. Yet the lifting $A \rightarrow \PX$ is given by a trivial line bundle, so
$$
\varphi_X^* L \otimes \sO_{\upA} \simeq
\sO_{\PP(T_X)}(1) \otimes \sO_{\upA} \simeq \sO_{\upA}.
$$ 
Since $L$ is an ample line bundle on $Y$,
we see that $\upA$ is contracted by $\varphi_X$ onto a point. 
Thus we have a contradiction to Lemma \ref{lemma-key}.
\end{proof}

\begin{remark} \label{remark-nef-and-big}
Let us conclude by indicating a variant of Theorem \ref{theorem-split} under the weaker  assumption that $T_X$ is nef and big. We claim that in this case a normal split submanifold $M \subset X$ is a Fano manifold with semiample tangent bundle.

\begin{proof}
The basepoint-free theorem  implies that $T_X$
is semiample in the sense of \cite{Fuj92}, cf. \cite[Prop.5.5]{MOSWW} for a proof. 
In particular we can define the birational morphism \eqref{define-varphiX} using the global sections of some positive multiple of $\sO_{\PX}(1)$, and Lemma \ref{lemma-key} 
holds for this morphism.

Since $M \subset X$ is normal split, its tangent bundle $T_M$ is also semi-ample. Moreover, by \cite[Main Thm.]{DPS94} there exists a finite \'etale cover $\holom{\eta}{M'}{M}$ such that $M'$
admits a smooth fibration $\holom{f}{M'}{A}$ onto an abelian variety such that the general fibre is Fano. Arguing by contradiction we assume that $A$ is not a point. Since $\eta$ is \'etale, the splitting map
$\holom{s_M}{T_X \otimes \sO_M}{T_M}$
lifts to splitting map 
$$
\holom{\tilde s_M}{\eta^* (T_X \otimes \sO_M)}{\eta^* T_M \simeq T_{M'}}.
$$
Since $T_{M'}$ is semiample \cite[Lemma 1]{Fuj92} the tangent map $\holom{\tau_f}{T_{M'}}{T_A}$ splits
by \cite[Cor.4]{Fuj92}. Thus we have $T_{M'} \simeq T_{M'/A} \oplus \sO_{M'}^{\oplus \dim A}$ and we can use a quotient line bundle $T_{M'} \twoheadrightarrow \sO_{M'}$ to define a lifting $M' \rightarrow \PX$ such that the image is contracted by $\varphi_X$ onto a point. Now the proof of Theorem \ref{theorem-split} yields a contradiction.
\end{proof}
\end{remark}


\newcommand{\etalchar}[1]{$^{#1}$}

\end{document}